\documentclass{amsart} 
\usepackage{fourier}
\usepackage{faktor}
\usepackage{amssymb}
\usepackage{amsmath}
\usepackage{amsthm}
\usepackage{stmaryrd}
\usepackage{hyperref}
\usepackage[all]{xy}
\usepackage[shortlabels]{enumitem}

\def\Q{\mathbb{Q}}
\def\Z{\mathbb{Z}}
\def\C{\mathbb{C}}

\def\F{\mathbb{F}}

\newcommand{\cB}{{\mathcal B}}
\newcommand{\cF}{{\mathcal F}}
\newcommand{\cG}{{\mathcal G}}
\newcommand{\cH}{{\mathcal H}}
\newcommand{\cK}{{\mathcal K}}
\newcommand{\cM}{{\mathcal M}}
\newcommand{\cO}{{\mathcal O}}
\renewcommand{\mod}{\bmod}

\newcommand{\set}[1]{\left\lbrace#1\right\rbrace }

\DeclareMathOperator{\Tr}{Tr}
\DeclareMathOperator{\Hom}{Hom}

\DeclareMathOperator{\Aut}{Aut}
\DeclareMathOperator{\ICM}{ICM}
\DeclareMathOperator{\Pic}{Pic}
\DeclareMathOperator{\End}{End}
\DeclareMathOperator{\GL}{GL}
\DeclareMathOperator{\rank}{rank}
\DeclareMathOperator{\Stab}{Stab}

\DeclareMathOperator{\Pol}{Pol}
\DeclareMathOperator{\pPol}{Pol^{1}}
\DeclareMathOperator{\Isom}{Isom}
\DeclareMathOperator{\Frob}{Frob}

\newcommand{\BassMod}[1]{{\cB}({#1})}
\newcommand{\vphi}{{\varphi}}

\DeclareMathOperator{\AV}{AV}

\newcommand{\AVord}[1]{\AV^{\text{ord}}({#1})}
\newcommand{\AVcs}[1]{\AV^{\text{cs}}({#1})}
\newcommand{\Modord}[1]{\cM^{\text{ord}}({#1})}
\newcommand{\Modcs}[1]{\cM^{\text{cs}}({#1})}

\newcommand{\Fcs}{\cF^{\text{cs}}}
\newcommand{\Ford}{\cF^{\text{ord}}}
\newcommand{\idcl}[1]{[{#1}]}

\renewcommand{\bar}{\overline}

\newtheorem{df}{Definition}[section]
\newtheorem{prop}[df]{Proposition}

\newtheorem{thm}[df]{Theorem}
\newtheorem{cor}[df]{Corollary}
\newtheorem{remark}[df]{Remark}
\newtheorem{example}[df]{Example}

\title{Computing abelian varieties over finite fields isogenous to a power}
\author{Stefano Marseglia}
\address{Matematiska institutionen, Stockholms universitet, Sweden}
\curraddr{Mathematical Institute, Utrecht University, The Netherlands}
\email{s.marseglia@uu.nl}

\begin{document}
\maketitle

\let\thefootnote\relax\footnote{
This is the accepted version of the following article:
\emph{
Stefano Marseglia;
Computing abelian varieties over finite fields isogenous to a power.
Res. Number Theory 5 (2019), no. 4, Paper No. 35, 17 pp.},
 which has been published (open access) in final form at
\url{https://doi.org/10.1007/s40993-019-0174-x}
}

\vspace{-1cm}
\begin{abstract}
In this paper we give a module-theoretic description of the isomorphism classes of abelian varieties $A$ isogenous to $B^r$, where the characteristic polynomial $g$ of Frobenius of $B$ is an ordinary square-free $q$-Weil polynomial, for a power $q$ of a prime $p$, or a square-free $p$-Weil polynomial with no real roots.
Under some extra assumptions on the polynomial $g$ we give an explicit description of all the isomorphism classes which can be computed in terms of fractional ideals of an order in a finite product of number fields.
In the ordinary case, we also give a module-theoretic description of the polarizations of $A$.
\end{abstract}


\section{Introduction}

It is well known that abelian varieties of dimension $g$ over the complex numbers can be functorially described by full lattices $L \subset \C^g$ and that such a description becomes an equivalence of categories when we only consider the lattices $L$ such that the associated torus $\C^g/L$ admits a Riemann form.
When we move to the wilder realm of positive characteristic we cannot have such a functorial description due to the existence of objects like supersingular elliptic curves whose endomorphisms form a quaternionic algebra which does not admit a $2$-dimensional representation, as pointed out by Serre.
Nevertheless, when we are working over a finite field $\F_q$, with $q$ a power of a prime $p$, we have analogous descriptions if we restrict ourselves to some subcategories  of the category of abelian varieties over finite fields.
More precisely, Deligne proved in \cite{Del69} that there is an equivalence between the category of ordinary abelian varieties over $\F_q$ and the category of finitely generated free $\Z$-modules with an endomorphism satisfying some easy-to-state axioms.
This description has been extended by Centeleghe and Stix in \cite{CentelegheStix15} for abelian varieties over the prime field $\F_p$ whose characteristic polynomial of Frobenius does not have real roots.
In the ordinary case, Howe has included in this equivalence the notions of dual variety and polarizations, see \cite{Howe95}.

In \cite{MarAbVar18} we have used such descriptions to produce algorithms to compute the isomorphism classes of abelian varieties with square-free characteristic polynomial of Frobenius and, when applicable, the polarizations and the corresponding automorphism groups.
The algorithms make use of the fact that the target category of Deligne's and Centeleghe-Stix functors is equivalent to a category of fractional ideals of a certain order in the \'etale algebra $\Q[x]/(h)$, where $h$ is the characteristic polynomial.

In the present paper we extend such a description to the case when the characteristic polynomial $h$ is a power of a square-free polynomial, say $h=g^r$.
Instead of fractional ideals we will have to consider lattices in $K^r$ with an $R$-modules structure, where $K=\Q[x]/(g)$ and $R=\Z[x,y]/(g(x),xy-q)$.
In the ordinary case we translate the notion of a polarization to this context.

When the order $R$ is Bass there is a classification of such modules, see \cite{basshy63} and \cite{LevyWiegand85}, and we can explicitly compute representatives of the isomorphism classes of the abelian varieties.

There are other categorical descriptions of the category of abelian varieties isogenous to a power of elliptic curves in terms of modules with extra-structure, see the Appendix in \cite{Lauter02}, \cite{Kani11} and \cite{JKPRSBT18}.
We do not make use of such descriptions and instead we work with Deligne's and Centeleghe-Stix equivalences because they allow us to deduce results also about powers of abelian varieties of dimension greater than $1$. 

The paper is structured as follows. 
In Section \ref{sec:orders} we recall the notion of an order and a fractional ideal, with a focus on Bass orders.
In Section \ref{sec:abvars} we describe the categorical equivalences that we are going to use in Section \ref{sec:powofbass}, where we focus on the case of abelian varieties with characteristic polynomial of the form $h=g^r$, with $g$ square-free.
These equivalences are based on the theorems of Deligne and Centeleghe-Stix cited above and the target category of the functors realizing them is well suited for computational purposes.

In Section \ref{sec:polarization} we translate the notion of a polarization into module-theoretic language.
Finally, in Section \ref{sec:examples} we apply our description and present the results of some computations.

The aim of the paper is to provide an effective algorithm to perform computations of isomorphism classes of abelian varieties.
The implementations can be found on the author's web-page.
We plan to use such algorithms to produce representatives that will be uploaded to the LMFDB \cite{lmfdb}.
Nevertheless the machinery built allows us to produce also theoretical results about the isogeny classes with characteristic polynomial of the form $g^r$. 
For example, if $R$ is a Bass order, we can prove that if $r>1$ then an abelian variety in such an isogeny class is not just isogenous but also isomorphic to a product of $r$ abelian varieties, see Corollary \ref{cor:abprod}.
We can also prove that given two abelian varieties $A$ and $B$ is such an isogeny class, there exists and integer $r$ such that $A^r$ and $B^r$ are isomorphic if and only if $A$ and $B$ have the same endomorphism ring, see Corollary \ref{cor:isom_powers}.
We get also statements about polarized abelian varieties, see Corollary \ref{cor:decomp} and Remark \ref{rmk:obs_to_dec}.
In a forthcoming paper we plan to use the machinery introduced to answer questions related to field extensions, for example, whether an abelian variety $A$ over $\F_q$ can be defined over a proper subfield of $\F_q$.

\subsection*{Acknowledgments}
The author would like to thank Jonas Bergstr\"om for helpful discussions and for sharing the data used to compute Example \ref{ex:freq}.
We are also grateful to Rachel Newton and Christophe Ritzenthaler for comments on a previous version of the paper, which is part of the author's Ph.D thesis \cite{MarPhDThesis18}.
We express our gratitude to the anonymous reviewers of Research in Number Theory for their useful comments and suggestions.

\subsection*{Conventions}
All rings considered are commutative and unital.
All morphisms between abelian varieties $A$ and $B$ over a field $k$ are also defined over $k$, unless otherwise specified.
In particular, we write $\Hom(A,B)$ for $\Hom_k(A,B)$.
Also, an abelian variety $A$ is simple if it is so over the field of definition.

\section{Orders}
\label{sec:orders}

Let $g$ be an integral square-free monic polynomial, say of degree $n$.
Let $K$ be the \'etale $\Q$-algebra $\Q[x]/(g)$.
Note that $K$ is a finite product of distinct number fields.
An \emph{order} $R$ in $K$ is a subring of $K$ whose additive group is isomorphic to $\Z^n$.
Among all orders in $K$ there exists a maximal one with respect to inclusion, which is called the \emph{maximal order} of $K$ and is denoted $\cO_K$.
An \emph{over-order} of $R$ is an order $S$ in $K$ containing $R$. Since the quotient $\cO_K/R$ is finite there are only finitely many over-orders of $R$.
A \emph{fractional ideal} of $R$ is a finitely generated sub-$R$-module of $K$ containing a non-zero-divisor.
Given two fractional $R$-ideals $I$ and $J$, we have that $I+J$, $I\cap J$,$IJ$, $(I:J)$ and $I^t$ are also fractional $R$-ideals.
Recall that the \emph{quotient ideal} $(I:J)$ and the \emph{trace dual ideal} $I^t$
are defined respectively as
\[ (I:J) = \set{ x \in K : xJ \subseteq I } \]
and
\[ I^t = \set{ x \in K : \Tr_{K/\Q}(xI)\subseteq \Z }. \]
Observe that the underlying additive subgroup of any fractional ideal $I$ is a free abelian group of rank $n$, that is, $I$ is a lattice in $K$.
Recall that if $I=\alpha_1\Z\oplus \ldots \oplus \alpha_n \Z $ then $I^t=\alpha^*_1\Z\oplus \ldots \oplus \alpha^*_n \Z$, where $\set{\alpha_i^*}_{i}$ is the dual basis characterized by the relations
$\Tr_{K/\Q}(\alpha^*_i\alpha_j)=1$ if $i=j$ and $0$ otherwise.

Given any full lattice $I$ in $K$ the set $(I:I)$ is an order.
If $I$ is a fractional $R$-ideal then $(I:I)$ will contain $R$.
This order is called the \emph{multiplicator ring} of $I$.
A fractional ideal $I$ is called \emph{invertible} if $I(S:I)=S$, where $S$ is the multiplicator ring of $I$.

An order $S$ is called \emph{Gorenstein} if every fractional ideal with multiplicator ring $S$ is invertible, or equivalently if $S^t$ is an invertible ideal, see \cite[Proposition 2.7]{buchlenstra}.
Examples of Gorenstein orders are $\cO_K$ and the monogenic order $R=\Z[x]/(f)$, see \cite[Example 2.8]{buchlenstra}.
An order $R$ is called \emph{Bass} if every over-order of $R$ is Gorenstein.
Since in this paper we will extensively use the properties of Bass orders we will list here other equivalent definitions.
\begin{prop}
\label{prop:bassorders}
   Let $R$ be an order. The following are equivalent:
   \begin{itemize}
      \item $R$ is Bass (every over-order is Gorenstein);
      \item every fractional $R$-ideal can be generated by $2$ elements;
      \item $R$ is a \emph{cyclic index} order, that is, the finite $R$-module $\cO_K/R$ is cyclic.
   \end{itemize}
\end{prop}
The study of such orders started with the paper \cite{basshy63} on Gorenstein rings.
There are many sources where one can find a proof of Proposition \ref{prop:bassorders} (and other characterizations), for example \cite[Theorem 2.1]{LevyWiegand85}. 
Since every fractional ideal of a quadratic order can be generated by $2$ elements as an abelian group, they are examples of Bass orders. 

Given an order $R$ we define the \emph{ideal class monoid} as
\[ \ICM(R) = \faktor{\set{ \text{fractional $R$-ideals} }}{\simeq} \]
and the \emph{ideal class group} as
\[ \Pic(R) = \faktor{\set{ \text{invertible fractional $R$-ideals} }}{\simeq}, \]
where the operations are induced by ideal multiplication. We will denote the class of the ideal $I$ by $\idcl{I}$.
Note that $\ICM(R) \supseteq \Pic(R)$ with equality if and only if $R=\cO_K$.
In general we have that
\[ \ICM(R) \supseteq \bigsqcup_S \Pic(S), \]
where the disjoint union is taken over the over-orders of $R$, with equality if and only if $R$ is Bass.
In particular, if this is the case, once we have a complete list of over-orders of $R$, it is easy to compute all the ideal classes of $R$, using the results from \cite{klupau05}.
For more about the computation of $\ICM(R)$, even in the non-Bass case, we refer to \cite{MarICM18}.

Recall that an $R$-module $M$ is \emph{torsion-free} if the canonical map $M\to M\otimes_R K$ is injective.
\begin{df}
    For an order $R$ in $K$, we define $\BassMod{r}$ as the category of torsion-free $R$-modules $M$ such that $M\otimes K$ is a free $K$-module of rank $r$.
    The morphisms are the $R$-linear morphisms.
\end{df}

Crucial for our purpose is the fact that, when $R$ is a Bass order, the modules in $\BassMod{r}$ can be written in a canonical form in terms of over-orders of $R$ and fractional ideals.

\begin{thm}
\label{thm:decompmodules}
    Let $R$ be a Bass order and let $M$ be in $\BassMod{r}$.
    Then there are fractional $R$-ideals $I_1,\ldots, I_r$ with $(I_1:I_1)\subseteq \ldots \subseteq (I_r:I_r)$ such that 
    \[ M\simeq I_1\oplus\ldots \oplus I_r. \]    
    The isomorphism class of $M$ is uniquely determined by the chain of over-orders $(I_i:I_i)$ and the isomorphism class $\idcl{I_1\cdots I_r}$.    
\end{thm}
This result was first proved in \cite[Theorem 1.7]{Ba62} for Noetherian integral domains with finite integral closure such that every ideal can be generated by $2$ elements. 
Later it was reproved with a different method in \cite[Theorem 8]{BF65} for an order in a commutative, separable, semisimple extension of the quotient field of a Dedekind domain.
In \cite[Theorem 7.1]{LevyWiegand85} the results was generalized to Bass rings, that is, commutative rings without nilpotents with finite integral closure such that every ideal can be generated by $2$ elements.

Using the same notation as in Theorem \ref{thm:decompmodules}, we see that $M$ can be written in a canonical form
\[ M \simeq S_1 \oplus \ldots \oplus S_{r-1} \oplus I, \]
with $S_1\subseteq \ldots \subseteq S_{r-1} \subseteq (I:I)$ where $S_i=(I_i:I_i)$ and $\idcl{I}=\idcl{I_1\cdots I_r}$.
Moreover, the chain of over-orders of $R$ together with $\idcl{I}$ uniquely determines $M$ up to isomorphism.
In particular, if we know all the over-orders of $R$ and their Picard groups we can easily compute representatives for all the isomorphism classes of modules in $\BassMod{r}$, for every $r$.
\begin{prop}
\label{prop:homs}
    Let $M = I_1 \oplus \ldots \oplus I_r \in \BassMod{r}$ and $N = J_1 \oplus \ldots \oplus J_s \in \BassMod{r}$.
    Then
    \[ \Hom_R(M,N) = \set{ A \in \cM_{s\times r}( K ) : A_{j,i} \in ( J_j : I_i ) }. \]    
\end{prop}
\begin{proof}
   The statement follows from the fact that $\Hom_R(I_i,J_j)=(J_j:I_i)$ since every $R$-linear morphism $\vphi:I_i\to J_j$ is a multiplication by $\alpha\in K$, where $\alpha$ is the image of $1_K$ under the induced $K$-linear endomorphism $\vphi\otimes \Q$ of $K$.
\end{proof}

In particular, for $M=S_1\oplus \ldots \oplus S_{r-1}\oplus I$ as above
we have
    \[ \End_R(M) = 
    \begin{pmatrix}
    S_1 	&  (S_1:S_2)		& \ldots		& (S_1:S_{r-1}) 	& (S_1:I) \\
    S_2 	& S_2 	   	 	& \ldots 	& (S_2:S_{r-1}) 	& (S_2:I) \\
    \vdots 	& \vdots   	& \ddots 	& \vdots  		& \vdots \\
    S_{r-1} 	& S_{r-1}	& \ldots 	& S_{r-1} 		& (S_{r-1}:I) \\
    I 	& I				& \ldots 	& I 			& (I:I)
    \end{pmatrix}
    \]
and 
\[\Aut_R(M)=\set{ A \in \End_R(M) \cap \GL_r(K) : A^{-1} \in \End_R(M) }.\]

If $R$ is a Bass order and $M$ and $N$ are two modules in $\BassMod{r}$, it is easy  using Theorem \ref{thm:decompmodules} to check whether they are isomorphic. If this is the case, it is possible to explicitly construct a matrix $A$ realizing the isomorphism, as the next example shows. 
\begin{example}[{\cite[Lemma 8]{BF65}}]
   Let $R$ be a Bass order and let $I_1$ and $I_2$ be fractional $R$-ideals with multiplicator rings $S_1$ and $S_2$, respectively, with $S_1 \subseteq S_2$.
   Then by the classification given in Theorem \ref{thm:decompmodules} we have an $R$-linear isomorphism
   \[ I_1\oplus I_2 \simeq S_1 \oplus (I_1I_2). \]
   We want to exhibit a matrix $A$ realizing the isomorphism.
   Since $I_1$ is invertible in $S_1$, there are elements $c_1$ and $c_2$ in $K^\times$ such that $c_1I_1+c_2I_2=S_1$, see \cite[Algorithm 1.3.14]{cohenadv00}.
   Thus there are $a_1 \in c_1I_1$ and $a_2 \in c_2I_2$ such that $1=a_1+a_2$.
   We claim that the matrix
   \[ A=
      \begin{pmatrix}
         c_1 & -c_2 \\[8pt]
         \dfrac{a_2}{c_2} & \dfrac{a_1}{c_1}
      \end{pmatrix}
   \]
   satisfies $A(I_1\oplus I_2) = S_1 \oplus I_1I_2$ (where the action is on column vectors).
   Indeed given $i_1\in I_1$ and $i_2\in I_2$ we have
   \[ A  
   	  \begin{pmatrix}
         i_1 \\[3pt]
         i_2
      \end{pmatrix} = 
      \begin{pmatrix}
         c_1i_1 - c_2i_2 \\[8pt]
         \dfrac{a_2}{c_2}i_1 + \dfrac{a_1}{c_1}i_2
      \end{pmatrix}. \]
    Observe that
    \[ c_1i_1 - c_2i_2\in c_1I_1+c_2I_2=S_1 \]
    and
    \[ \dfrac{a_2}{c_2}i_1 + \dfrac{a_1}{c_1}i_2 \in I_1I_2, \]
    as required.
    Furthermore, note that $\det(A)=1$ and that given $s_1\in S_1$ and $i_1i_2\in I_1I_2$ we have
    \[ A^{-1}  
   	  \begin{pmatrix}
         s_1 \\[3pt]
         i_1i_2
      \end{pmatrix} = 
      \begin{pmatrix}
        \dfrac{a_1}{c_1} & c_2\\[8pt]
        -\dfrac{a_2}{c_2} & c_1       
      \end{pmatrix}
      \begin{pmatrix}
         s_1 \\[3pt]
         i_1i_2
      \end{pmatrix} =
      \begin{pmatrix}
         \dfrac{a_1}{c_1}s_1 + c_2i_2i_1\\[8pt]
         -\dfrac{a_2}{c_2}s_1 + c_1i_1i_2
      \end{pmatrix}. \]
	Since $a_1/c_1\in I_1$, $a_2/c_2\in I_2$, $c_1i_1,c_2i_2 \in S_1$ and $I_1I_2$ is additively generated by elements of the form $i_1i_2$, we conclude that $A^{-1}(S_1\oplus I_1I_2)\subset I_1\oplus I_2$.
\end{example}


\section{The category of abelian varieties over a finite field}
\label{sec:abvars}

Let $q$ be a power of a prime number $p$ and let $\AV(q)$ be the category of abelian varieties defined over $\F_q$. 
For $A$ in $\AV(q)$ consider the induced action of the Frobenius endomorphism on the $l$-adic Tate modules $T_lA$, for any prime $l\neq p$, and let $h_A$ be the corresponding characteristic polynomial.
Then $h_A$ is a \emph{$q$-Weil polynomial}, that is, a monic polynomial in $\Z[x]$ of even degree with roots of complex absolute value $\sqrt{q}$.
In particular $h_A$ has degree $2\dim(A)$ and uniquely determines the \emph{isogeny class} of $A$, in the sense that an abelian variety $B$ is isogenous to $A$ if and only if $h_A = h_B$.

By the Poincar\'e Decomposition Theorem we know $A$ is isogenous to a product
\[ A\sim B_1^{e_1}\times \ldots \times B_r^{e_r},\]
where $e_i$ are positive integers and the $B_i$'s are simple pairwise non-isogenous abelian varieties. It follows that
\[h_A = h_{B_1}^{e_1}\cdots h_{B_r}^{e_r}. \]
Recall that for a simple abelian variety $B$ in $\AV(q)$ the polynomial $h_B$ is a power of an irreducible polynomial, say $m^a$, and the exponent $a$ is uniquely determined by the $p$-adic factorization of $m$, see \cite[Theorem 8]{MilWat71}.

Using this recipe, we can list all \emph{characteristic polynomials} $h$ of the Frobenius of abelian varieties over a finite field $\F_q$ of a given dimension $g$, for example see \cite{Hal10} for $g=3$ and \cite{HalSin12} for $g=4$. By Honda-Tate theory, see \cite{Tate66} and \cite{Honda68}, this corresponds to describing all isogeny classes of abelian varieties in $\AV(q)$ of a given dimension $g$.
For such a polynomial $h$, denote by $\AV(h)$ the full subcategory of $\AV(q)$ whose objects are the abelian varieties in the isogeny class determined by $h$.

We will restrict our attention to two subcategories of $\AV(q)$. Recall that an abelian variety $A$ over $\F_q$ is called \emph{ordinary} if exactly half of the roots of $h_A$ over $\bar{\Q}_p$ are $p$-adic units. There are many other characterizations of ordinary abelian varieties. For example see \cite[Section 2]{Del69}. We will denote the full subcategory of $\AV(q)$ consisting of ordinary abelian varieties by $\AVord{q}$.
We will also consider the subcategory $\AVcs{p}$ of abelian varieties $A$ over the prime field $\F_p$ such that $h_A$ has no real roots, that is, $h_A(\sqrt{p}) \neq 0$.
We will give functorial descriptions of $\AVord{q}$ and $\AVcs{p}$ in terms of $\Z$-lattices with extra structure.
More precisely, consider the following categories: 
\begin{itemize}
   \item the category $\Modord{q}$ consisting of pairs $(T,F)$ where $T$ is a free finitely generated $\Z$-module and $F$ is a $\Z$-linear endomorphism of $T$ such that
   \begin{itemize}
      \item the action of $F\otimes \Q$ on $T\otimes_\Z \Q$ is semisimple;
      \item the eigenvalues of $F\otimes \Q$ have complex absolute value $\sqrt{q}$;
      \item half of the roots of the characteristic polynomial of $F\otimes \Q$ over $\bar{\Q}_p$ are units;
      \item there exists an endomorphism $V$ of $T$ such that $FV=q$;
   \end{itemize}
   \item the category $\Modcs{p}$ consisting of pairs $(T,F)$ where $T$ is a free-finitely generated $\Z$-module and $F$ is a $\Z$-linear endomorphism of $T$ such that
   \begin{itemize}
      \item the action of $F\otimes \Q$ on $T\otimes_\Z \Q$ is semisimple;
      \item the eigenvalues of $F\otimes \Q$ have complex absolute value $\sqrt{p}$;
      \item the characteristic polynomial of $F\otimes \Q$ has no real roots;
      \item there exists an endomorphism $V$ of $T$ such that $FV=p$.
   \end{itemize}
\end{itemize}
In both categories, a morphism $(T,F) \to (T',F')$ is a $\Z$-linear morphism $\vphi:T\to T'$ inducing a commutative diagram
\[ \xymatrix{
  T \ar[d]_F \ar[r]^\vphi & T'\ar[d]^{F'}\\
  T \ar[r]^\vphi          & T'} \]

The main tools to understand the categories $\AVord{q}$ and $\AVcs{p}$ are given in the following theorem. 
\begin{thm}
\label{thm:equiv}
   There are equivalences of categories
   \[ \Ford : \AVord{q} \to \Modord{q} \]
   and
   \[ \Fcs : \AVcs{p} \to \Modcs{p}. \]
   If $A\mapsto (T,F)$, then $\rank_\Z T = 2\dim A$ and the Frobenius endomorphism $\Frob_A$ is sent to $F$, so in particular they have the same characteristic polynomial.
\end{thm}
\begin{proof}
   For the ordinary case over $\F_q$ see \cite[Section 7]{Del69}.
   For the case with no-real roots over $\F_p$ we use the covariant version \cite[1.7]{CentelegheStix15} of the equivalence given in \cite[Theorem 1]{CentelegheStix15}.
\end{proof}
\begin{remark}
We remark that the functors of Theorem \ref{thm:equiv} depend on some choices, which can be made in a way that $\Fcs$ extends $\Ford$ on $\AVord{p}$, see \cite[7.4]{CentelegheStix15}.
\end{remark}


\section{Abelian varieties isogenous to a power}
\label{sec:powofbass}

Let $h$ be a characteristic polynomial of an abelian variety in $\AVord{q}$ or $\AVcs{p}$.
Assume moreover that $h=g^r$ for some square-free polynomial $g$ in $\Z[x]$.
Put $K=\Q[x]/(g)$ and $\alpha=x \mod (g)$.
Denote with $R$ the order $\Z[\alpha,q/\alpha]$ in $K$ (with $q=p$ if we are in $\AVcs{p}$).
Observe that the order $R$ is Gorenstein, see \cite[Theorem 11]{CentelegheStix15}.

\begin{thm}
\label{thm:eqideals}
\begin{enumerate}[(a)]
   \item \label{thm:eqideals:a} If $\AV(h)\subset \AVord{q}$ or $\AV(h)\subset \AVcs{p}$ then there is an equivalence of categories $\cF:\AV(h) \to \BassMod{r}$.
   \item \label{thm:eqideals:b} If $R$ is a Bass order, $\cF$ induces a bijection between
   \[ \faktor{\AV(h)}{\simeq} \]
   and the set of pairs
   \[ (S_1\subseteq S_2 \subseteq \ldots \subseteq S_r , \idcl{I} ), \]
   where each $S_i$ is an over-order of $R$ and $\idcl{I}$ denotes the isomorphism class of a fractional ideal with $(I:I)=S_r$.
\end{enumerate}
\end{thm}
\begin{proof}

   Denote by $\cM(h)$ the image of $\AV(h)$ via $\Ford$ (or $\Fcs$).
   We will define an equivalence $\cG:\cM(h) \to \BassMod{r}$.
   Take $A$ in $\AV(h)$ and let $(T,F)$ be the image of $A$ in $\cM(h)$ via $\Ford$ (or $\Fcs$).
   The minimal polynomial of the $\Q$-linear endomorphism $F\otimes \Q$ of $T\otimes \Q$ is $g$.
   So by definition of $\Modord{q}$ (or $\Modcs{p}$) we have that $F$ and $V$ induce on $T$ an $R$-module structure via the isomorphism $R\simeq \Z[F,V]$ given by $\alpha\mapsto F$.
   Denote this $R$-module by $M$ and put $\cG((T,F))=M$.
   Observe that the action of $F$ on $T$ is faithful, since it becomes multiplication by $q$ (or by $p$) after composing with $V$, and hence $M$ is torsion-free.
   Let's prove that $M\otimes_R K$ is a free $K$-module of rank $r$.      
   Since $g$ is square-free, it is a product of distinct irreducible polynomials, say $g=g_1\cdots g_s$.
   In particular, $K$ is isomorphic to the product of number fields $\prod_i K_i$, where $K_i=\Q[x]/(g_i)$. Let $e_i$ be the image in $K$ of the multiplicative unit of $K_i$ under this isomorphism, so that $1_K=e_1+\ldots +e_s$ and $Ke_i\simeq K_i$ for each $i$. Hence
   \[ M \otimes_R K = M \otimes_R \left( \bigoplus_{i=1}^s Ke_i\right) \simeq \bigoplus_{i=1}^s \left(M\otimes_R Ke_i \right). \]
   Since the action of $F\otimes \Q$ is semisimple, there is a direct sum decomposition $T\otimes_\Z \Q = W_1\oplus \ldots \oplus W_s$ such that the action of $F\otimes \Q$ on each $W_i$ is simple.
   This means that, possibly after renumbering, we can assume that the minimal polynomial of $F \otimes\Q|_{W_i}$ is $g_i$ and so 
   $\dim_\Q W_i = r \deg(g_i)$.
   In particular the action of $F\otimes \Q$ on $W_i$ is the same as the action of $\alpha$ on $M\otimes_R Ke_i$.
   Since $\deg(g_i)= \dim_\Q Ke_i$ it follows that $\dim_{Ke_i}(M\otimes_R Ke_i) =r $ and hence, by taking the direct sum over $i$, we obtain an isomorphism
   \[ M\otimes_R K \simeq K^r. \]
   Therefore $M$ is in $\BassMod{r}$.
   It is clear by construction that $\cG$ is a fully faithful and essentially surjective functor.
   Define $\cF$ as the composition of the equivalences $\Ford$ (or $\Fcs$) and $\cG$. In particular $\cF$ is an equivalence as well and we have concluded the proof of part~$\ref{thm:eqideals:a}$.
   Part $\ref{thm:eqideals:b}$ now follows directly from Theorem \ref{thm:decompmodules}.
\end{proof}
\begin{remark}
	The equivalence $\cF$ of part \ref{thm:eqideals:a}  of Theorem \ref{thm:eqideals} is compatible with products.
	More precisely, if we denote $\cF_i$ the equivalence $\AV(g^i)\to \BassMod{i}$ then we pick $A$ and $B$ respectively in $\AV(g^m)$ and $\AV(g^n)$ then we have a canonical isomorphism 
	\[  \cF_{m+n}(A\times B)\simeq \cF_{m}(A)\oplus\cF_{n}(B) \in \BassMod{m+n}.\]
	We will denote all these functors with $\cF$.
\end{remark}
\begin{cor}
\label{cor:abprod}
   Assume that $R$ is a Bass order. Then every abelian variety $A$ in $\AV(h)$ is isomorphic to
   \[ B_1 \times \ldots \times B_r, \]
   for some abelian varieties $B_i$ in $\AV(g)$.
\end{cor}
\begin{proof}
   Put $M=\cF(A)$ by Theorem \ref{thm:eqideals}. By Theorem \ref{thm:decompmodules} we have that there are fractional ideals $I_1,\ldots,I_r$ such that 
   \[ M \simeq I_1\oplus\ldots \oplus I_r. \]
   Again by Theorem \ref{thm:eqideals}, we get that there are abelian varieties $B_i$ in $\AV(g)$ such that $B_i=\cF(I_i)$ for each $i=1,\ldots,r$.
\end{proof}
\begin{remark}
   In Corollary \ref{cor:abprod}, the abelian varieties $B_i$ are simple if and only if $g$ is irreducible.
   This follows from \cite[Theorem 3.3]{Howe95} for the ordinary case over $\F_q$ and from \cite[Theorem 8]{MilWat71} for characteristic polynomials over $\F_p$ with no real roots.
\end{remark}

\begin{cor}
   Let $A$ be in $\AV(h^r)$. If $r>1$ then $\End(A)$ is not commutative.
\end{cor}
\begin{proof}
   It follows from the fact that $\End(A)\otimes \Q \simeq \cM_{r\times r}(K)$, see Proposition \ref{prop:homs}.
\end{proof}

\begin{cor}
\label{cor:isom_powers}
   Assume that $R$ is a Bass order.
   Let $A$ and $B$ be in $\AV(g)$.
   Then there exists a positive integer $r$ such that $A^r$ and $B^r$ are isomorphic if and only if $\End(A)=\End(B)$.
   If this is the case, then $r$ is bounded by the exponent of $\Pic(R)$.
\end{cor}
\begin{proof}
   Put $I=\cF(A)$ and $J=\cF(B)$.
   Assume first that $A^r \simeq B^r$.
   Then using Theorem \ref{thm:eqideals} we have that 
   \[ \bigoplus_{i=1}^rI \simeq_R \bigoplus_{i=1}^rJ \]   
   which is equivalent to having $I^r\simeq J^r$ and $(I:I)=(J:J)$.
   In particular this last condition give us $\End(A)=\End(B)$.
   Conversely, assume that the endomorphism rings of $A$ and $B$ are the same and put $S=\cF(\End(A))$.
   Let $r$ be the exponent of $\Pic(S)$, so that $I^r\simeq J^r \simeq S$.
   In particular, using the same argument as before in the opposite direction we obtain that $A^r\simeq B^r$.
   The last statement follows from the fact that since $S$ is an over-order of $R$, there is a surjective map from $\Pic(R)$ to $\Pic(S)$ and in particular the exponent of $\Pic(S)$ divides the exponent of $\Pic(R)$.
\end{proof}

\begin{remark}
   Note that Theorem \ref{thm:eqideals} is a generalization of \cite[Theorem 4.3]{MarAbVar18}, where we deal with the case when $h$ is square-free, that is, $r=1$.
\end{remark}
\begin{remark}
   Theorem \ref{thm:decompmodules} tells us that if $R$ is a Bass order, then every torsion-free $R$-module of finite rank is isomorphic to a direct sum of fractional $R$-ideals. The reverse implication does not hold. In \cite{basshy63}, the author describes when it fails, but overlooks some cases.
   The gaps were filled in \cite{NazRou67} and in \cite{Haef90} in the local case and in \cite{HaefLevy88} it is described how to go from the local case to the global case.
   We have not analyzed if those exceptions could arise from orders generated by Weil polynomials, which could potentially extend our description to more isogeny classes.
\end{remark}


\section{Polarizations}
\label{sec:polarization}

In this section we will continue using the same notation as in Section \ref{sec:powofbass}, but we will restrict to the case when $h$ is ordinary.
Our goal is to describe  what the polarizations of an abelian variety $A$ in $\AV(h)$ correspond to in the category $\BassMod{r}$ via the equivalence $\cF$ of Theorem \ref{thm:eqideals}.$\ref{thm:eqideals:b}$.

Note that $K$ is a CM-algebra, that is, there is an involution $a\mapsto \bar a$ that acts as complex conjugation after composing with any non-zero homomorphism $\vphi:K \to \C$.
In particular, we have that $R=\Z[\alpha,\bar\alpha]$ where $\bar\alpha = q/\alpha$.
Observe that the homomorphisms $K\to \C$ come in conjugate pairs.
We call a choice of half of these homomorphisms, one for each conjugate pair, a \emph{CM-type} of $K$. 
For every $R$-module $M$ in $\BassMod{r}$, since we can identify $M$ with a sub-$R$-module of $K^r$, we have an induced action $M\mapsto \bar M$.
Moreover, if we consider $M$ as a submodule of $K^r$, we see that the trace $\Tr_{K/\Q}:K\to \Q$ induces a non-degenerate bilinear form $\Tr$ on $M$ by
\[ \Tr:M\times M \to \Q, \qquad ((x_i)_{i=1}^r,(y_j)_{j=1}^r)\mapsto \sum_{i=1}^r \Tr_{K/\Q}(x_iy_i),  \]
where we think of all vectors in $K^r$ as column vectors.
In analogy to the $r=1$ case, when $M$ is a fractional $R$-ideal, we define the \emph{trace dual} $M^t$ of $M$ to be the dual module with respect to $\Tr$.
In particular, if $n=\deg(h)$ and we fix a $\Z$-basis
\[ M=\alpha_1\Z\oplus\ldots\oplus\alpha_{nr}\Z,\quad \text{with }\alpha_j \in K \]
then we can write
\[ M^t=\alpha^*_1\Z\oplus\ldots\oplus\alpha^*_{nr}\Z, \]
where $\alpha_i^*$ is the dual basis characterized by $\Tr_{K/\Q}(\alpha_i\alpha_j^*)=1$ if $i=j$ and $0$ otherwise.

\begin{prop}
\label{prop:dualmodule}
Let $A$ be an abelian variety in $\AV(h)$ and put $M=\cF(A)$.
If $A^\vee$ is the dual abelian variety of $A$, then there is a canonical isomorphism between $\cF(A^\vee)$ and $M^\vee$, where $M^\vee=\bar{M}^t$.
In particular, if $M=I_1\oplus \ldots \oplus I_r$ then $M^\vee=\bar{I}_1^t\oplus \ldots \oplus \bar{I}_r^t$.
\end{prop}
\begin{proof}
Let $\cG$ be the functor defined in the proof of Theorem \ref{thm:eqideals}.$\ref{thm:eqideals:a}$.
Put $(T,F)=\Ford(A)$, so that $\cG((T,F))=M$.
Following \cite[Proposition 4.5]{Howe95}, 
we have that $\Ford(A^\vee) = (T^\vee,F^\vee)$, where $T^\vee=\Hom_\Z(T,\Z)$ and $F^\vee(\psi) = \psi \circ V$ for every $\psi \in T^\vee$.
To conclude, we need to show that $\cG(T^\vee,F^\vee) = M^\vee$.
It is clear that $\cG$ sends $T^\vee$ to $\Hom_\Z(M,\Z)$ (as abelian groups).
Since the action of $F^\vee$ on $T^\vee$ is ``pre-composition with $V$'' and $\bar V =F$, we see that it will correspond, via $\cG$, to the multiplication by $\alpha$ after taking the complex conjugate.
More precisely, write $M=\alpha_1\Z\oplus\ldots\oplus\alpha_{nr}\Z$, for $\alpha_j\in K$, with $n=[K:\Q]$, and consider the $\Z$-linear isomorphism
\begin{align*}
\rho:\Hom_\Z(M,\Z) & \to \bar{M}^t \\
\psi & \longmapsto \sum_{i=1}^{nr}\psi(\alpha_i)\bar\alpha_i^*
\end{align*}
with inverse 
\[ \Tr(\bar{x}^T , -) \longmapsfrom x, \]
where $\bar{x}^T$ is the transpose of $\bar{x}$.
Using this identification, the pre-composition with $\bar \alpha$ on $\Hom_\Z(M,\Z)$ will correspond to multiplication by $\alpha$ on $\bar{M}^t$, where, as usual, $\alpha$ is the image of the Frobenius endomorphism via the functor $\cF$.
\end{proof}

\begin{remark}
From now on we will fix a basis of $K^r$ and fix an isomorphism $M\otimes K\simeq K^r$ for each $M\in \BassMod{r}$.
In particular this will allow us to represent morphisms (in a non-canonical way) in $\BassMod{r}$ as matrices with entries in $K$.
\end{remark}

\begin{cor}
\label{cor:dualmorph}
	Let $\mu:A \to B$ be a morphism of abelian varieties in $\AV(h)$.
	Put $\cF(A)=M$, $\cF(B)=N$ and $\cF(\mu)=\Lambda:M\to N$ in $\BassMod{r}$.
	Then $\mu$ is an isogeny if and only if $\det \Lambda \in K^\times$.
	Moreover, the dual morphism $\mu^\vee:B^\vee \to A^\vee$ corresponds via $\cF$ to the morphism $\Lambda^\vee=\bar\Lambda^T:N^\vee\to M^\vee$, where $\bar\Lambda^T$ is the transpose of $\bar\Lambda$.
\end{cor}
\begin{proof}
    Put $(T,F)=\Ford(A)$, $(T',F')=\Ford(B)$ and $\Ford(\mu)=(T,F)\overset{\lambda}{\to} (T',F')$.    
    Then $\mu$ is an isogeny if and only if the induced morphism $\lambda\otimes \Q$ is invertible.
    Let $\cG$ be the functor defined in the proof of Theorem \ref{thm:eqideals}.$\ref{thm:eqideals:a}$.
    Observe that $\cG(\lambda)=\Lambda$ and hence $\lambda\otimes \Q$ is invertible if and only if the matrix $\Lambda$ is invertible over $K$.
    
    Put $\Ford(\mu^\vee)=\lambda^\vee$ where $(T'^\vee,F'^\vee)\overset{\lambda^\vee}{\to} (T^\vee,F^\vee)$ is
    defined by $\lambda^\vee(\psi)=\psi\circ \lambda $ for every $\psi\in T'^\vee$.
    Using Proposition \ref{prop:dualmodule}, we see that, if $\cG(\lambda)=\Lambda$ then $\cG(\lambda^\vee)=\bar\Lambda^T$.
\end{proof}

In order to describe the polarizations we need a particular kind of CM-type which, roughly speaking, detects the complex structure ``coming from characteristic $p$'' on a pair $(T,F)$ in $\Modord{q}$.
More precisely, put
\[ \Phi=\set{ \vphi \in \Hom(K,\C) : v_p(\vphi(\alpha)) > 0 }, \]
where $v_p$ is the $p$-adic valuation induced by a fixed isomorphism $\bar\Q_p \simeq \C$.
In \cite{MarAbVar18} we give an algorithm to compute such a $\Phi$.
Recall that an element $a$ in $K$ is called \emph{totally imaginary} if $\bar a = -a$.
We say that a totally imaginary $a$ is \emph{$\Phi$-non-positive} if $\Im(\vphi(a)) \leq 0$ for every $\vphi$ in $\Phi$.

Observe that in $\Modord{q}$, an isogeny $\lambda:(T,F) \to (T^\vee,F^\vee)$ induces a bilinear form 
\[  b:T\times T \to \Z \qquad b(s,t)=\lambda(t)(s). \]
Then there exists a unique $K$-sesquilinear form $S$ on $T \otimes \Q$ such that $b = \Tr_{K/\Q}\circ S$ and,
using \cite[Proposition 4.9]{Howe95}, we have that $\mu$ is a polarization if and only if the associated $S$ is skew-Hermitian and for every $a$ in $K$ the element $S(a,a)$ is $\Phi$-non-positive.

\begin{thm}
\label{thm:pols}
    Let $A$ be an abelian variety in $\AV(h)$ and let $\mu:A\to A^\vee$ be an isogeny.
    Put $\cF(\mu)=\Lambda:M\to M^\vee $.    
    Then $\mu$ is a polarization if and only if
    \begin{itemize}
       \item $\Lambda=-\bar\Lambda^T$ and,
       \item for every column vector $a$ in $K^r$, the element
    $a^T\bar\Lambda \bar a $
    is $\Phi$-non-positive.  
    \end{itemize}    
    We have $\deg \mu = [M^\vee : \Lambda M]$.    
\end{thm}
\begin{proof}
    Put $\Ford(A)=(T,F)$.
    Using the functor $\cG$ from the proof of Theorem \ref{thm:eqideals}.$\ref{thm:eqideals:a}$ we can identify $T\otimes_\Z \Q$ with $K^r$ and, by abuse of notation, we will denote also by $b$ and $S$ the forms on $M$ and $M\otimes \Q$ induced by $\cG$.
    Let $m$ be a (column) vector in $M$ (seen as a submodule of $K^r$).
    Composing $\Lambda$ with the inverse of the isomorphism $\rho^{-1}$ introduced in the proof of Proposition \ref{prop:dualmodule}, we obtain
    \begin{align*}
       & M \overset{\Lambda}{\longrightarrow} M^\vee \overset{\rho^{-1}}{\longrightarrow} \Hom_\Z\left(M,\Z\right) \\
       & m \longmapsto \Lambda m \longmapsto \Tr((\bar{\Lambda m})^T , -)=\Tr(\bar{m^T\Lambda^T} , -).
    \end{align*}
    So we deduce that the bilinear form $S$ is given by
    \[ S(a,a') = (\bar{a'})^T\bar{\Lambda}^T a \]
    where $a$ and $a'$ are column vectors in $K^r$.
    Thus $S$ is skew-Hermitian if and only if $S(a,a')$ equals
    \[ -\bar{S(a',a)} = -\bar{ (\bar{a})^T\bar{\Lambda}^T a'} = -(\bar{a'}^T\Lambda a)  \]
    for arbitrary $a$ and $a'$, which is equivalent to $\Lambda = -\bar\Lambda^T$.\\
    The second condition follows directly from this description.
	The statement about the degree follows from \cite[Proposition 4.14]{Howe95}.
\end{proof}

Let $(M,\Lambda)$ and $(M',\Lambda')$ be the modules corresponding to two polarized abelian varieties.
A \emph{morphism of polarized abelian varieties} will be a morphism $\Psi:M \to M'$ satisfying
\[\bar{\Psi}^T\Lambda'\Psi=\Lambda,\]
since the dual morphism $\Psi^\vee $ is $ \bar{\Psi}^T$ by Corollary \ref{cor:dualmorph}.
Denote by $\Pol(M)$ the set of polarizations of $M$.

\begin{thm}
\label{thm:isompols}
   There is a degree-preserving action of $\Aut(M)$ on the set $\Pol(M)$ given by
   \begin{align*}
      \Aut(M) \times \Pol(M) & \longrightarrow \Pol(M) \\
      (U ,\Lambda) & \longmapsto \bar U^T \Lambda U
   \end{align*}
   Two polarizations of $M$ give rise to isomorphic polarized abelian varieties if and only if they lie in the same orbit.
   In particular, given a polarization $\Lambda$ on $M$, we have
   \[ \Aut(M,\Lambda) = \Stab( \Lambda ). \]
\end{thm}
\begin{proof}
   Let $\Lambda$ be a polarization of $M$ and $U$ an automorphism of $M$.
   Observe that the first condition of Theorem \ref{thm:pols} is satisfied for $\bar U^T \Lambda U$, since 
   \[ -(\bar{\bar U^T \Lambda U})^T = -(U^T\bar{\Lambda}\bar{U})^T = -\bar{U}^T(U^T\bar{\Lambda})^T = -\bar{U}\bar{\Lambda}^TU = \bar{U}\Lambda U, \]
   where the last equality holds because $-\bar{\Lambda}^T=\Lambda$.
   Note that given $a \in K^r$ we have that
   \[ a^T \bar{\bar{U}^T \Lambda U} \bar{a} = (Ua)^T\bar{\Lambda}\bar{(Ua)}\]
   which is then $\Phi$-non-negative, since $U$ is also an automorphism of $K^r$.
   Hence the second condition of Theorem \ref{thm:pols} holds as well and $\bar U^T \Lambda U$ is a polarization of $M$.
   The statement on the degree follows from the existence of $R$-linear isomorphisms
   \[ \dfrac{M^\vee}{(\bar U^T \Lambda U) M } \simeq \dfrac{(\bar U^T)^{-1}M^\vee}{(\Lambda U) M }\simeq \dfrac{M^\vee}{\Lambda M }. \]
\end{proof}

\begin{remark}
   Let $\pPol(M)$ be the subset of $\Pol(M)$ consisting of principal polarizations.
   Since the action of $\Aut(M)$ on $\Pol(M)$ is degree-preserving, we get an induced action on $\pPol(M)$. 
   Recall that an abelian variety defined over a finite field admits only finitely many non-isomorphic principal polarizations, or, in other words, the quotient 
   \[ Q=\faktor{\pPol(M)}{\Aut(M)} \]
   is finite. 
   Moreover, the action of $\Aut(M)$ on $\pPol(M)$ can be extended to the set $\Isom(M,M^\vee)$ of isomorphisms from $M$ to $M^\vee$.
   In particular, by fixing an element $A_0$ in $\Isom(M,M^\vee)$, we get 
   \[ \Isom(M,M^\vee)=\set{ A_0V : V \in \Aut(M) }. \]
   This suggests that a good understanding of $\Aut(M)$ will most likely allow us to handle $Q$, but if $r>1$, then $\Aut(M)$ is an infinite non-abelian group, making the situation computationally difficult, even if we were able to produce a (finite) set of generators.
\end{remark}

Recall that a polarized abelian variety $(A,\mu)$ is called \emph{decomposable} if there are proper subvarieties $B_1$ and $B_2$ of $A$, admitting polarizations $\beta_1$ and $\beta_2$, respectively, such that
\[ (A,\mu) \simeq (B_1\times B_2,\mu_1\times \mu_2). \]

\begin{cor}
\label{cor:decomp}
    Let $(M,\Lambda)$ in $\BassMod{r}$ correspond to a polarized abelian variety $(A,\mu)$. Then $(A,\mu)$ is decomposable if and only if there are an integer $m>1$ and polarized modules
	$(M_i,\Lambda_i) \in \BassMod{r_i}$ for $i=1,\ldots,m$ with $r_1+\ldots+r_m=r$
    and an isomorphism 
    \[ P: M_1\oplus \ldots \oplus M_m \to M \]
    satisfying
    \[\bar{P}^T(\Lambda_1\oplus\ldots\oplus\Lambda_m)P = \Lambda. \]
\end{cor}

\begin{proof}
	Let $\cF(A,\mu)=(M,\Lambda)$ and $\cF(B_i,\mu_i)=(M_i,\Lambda_i)$ for $i=1,\ldots,m$.
	Then there exists a polarized isomorphism $f:\prod_i(B_i,\mu_i) \to (A,\mu)$ if and only if there exists an $R$-linear map $P$ as in the statement of the corollary.
\end{proof}
\begin{remark}
\label{rmk:obs_to_dec}
Assume that $R$ is Bass and that $r>1$. 
For simplicity, let us assume that $g$ is irreducible and let $(A,\mu)$ be a polarized abelian variety in $\AV(g^r)$.
By Corollary \ref{cor:abprod} we know that $A$ is isomorphic to the product of $r$ simple abelian varieties.
Hence Corollary \ref{cor:decomp} tells us that the obstruction for $(A,\mu)$ to be decomposable is a property of the polarization $\mu$.
\end{remark}
The next example shows that a polarized module $(M,\Lambda)$ can be decomposable even if there is no way to put $\Lambda$ into a block diagonal matrix by the action of an element of $\Aut(M)$.
\begin{example}
      Let $K=\Q(F)$ be the number field generated by the $4$-Weil polynomial $h=x^2-x+4$.
      The order $\cO=\Z[F,4/F]=\Z+F\Z$ is maximal in $K$ and it has Picard group of order $2$.
      Put $I=2\Z+F\Z$.
      One can check that the $\cO$-ideal $I$ is not principal and hence represents the non-trivial ideal class of $\cO$.
      The previous discussion implies that there are $2$ isomorphism classes of elliptic curves, corresponding to $\cO$ and $I$, in that isogeny class.
      Let $y=\frac{1}{15}(-1+2F)$ and $z=\frac{1}{30}(-1+2F)$ be the principal polarizations of $\cO$ and $I$, respectively, that is, $y\cO=\bar{\cO}^t$ and $zI=\bar{I}^t$. 
      Now consider the abelian surfaces $\cO\oplus \cO$ and $I \oplus I$ and the following matrices
      \[P_0 = \begin{pmatrix}
               1 & \frac{-3-3F}{2}\\
	       \frac{-3-3F}{2} & \frac{-13+13F}{2}
              \end{pmatrix},\quad
	M = \begin{pmatrix}
               4 & 2\bar F -1 \\
               2 F -1 & 4
            \end{pmatrix},
      \]
      \[      
        D = \begin{pmatrix}
               y & 0  \\
               0 & y
            \end{pmatrix},\quad
        D' = \begin{pmatrix}
               z & 0  \\
               0 & z
            \end{pmatrix}.\]            
      Note that $M$ is a unimodular Hermitian matrix in $\GL_2(\cO)$ and hence $DM$ is a principal polarization on $\cO\oplus \cO$.
      Also, observe that the matrix $P_0$ represents an isomorphism
      \[I\oplus I \to \cO\oplus \cO,\]
      and that every such isomorphism can be described as $AP_0$ for some $A \in \GL_2(\cO)$.
      
      One can verify by using results contained in \cite{GHR18} that the polarization $DM$ is not the pullback of the product polarization $D$ of $\cO\oplus \cO$, that is, there is no matrix $B\in \GL_2(\cO)$ such that
      \[ \bar{B}^TDM B = D. \]
      On the other hand $DM$ is the pullback of the product polarization $D'$ of $I\oplus I$.
      Indeed we have
      \[ (\bar{AP_0})^TDM AP_0 = D', \]
      for 
      \[ A=\begin{pmatrix}
	   7-10F & -3-2F\\
	   -23+3F & -4+3F  
         \end{pmatrix} \in GL_2(\cO).
      \]
      Again, the matrix $A$ has been computed using results from \cite{GHR18}.
   \end{example}

\section{Examples}
\label{sec:examples}
With the previous results we can create effective algorithms to
compute examples of different phenomena of abelian varieties over
finite fields that are isogenous to powers.

\begin{example}
   In this example we compute all the isomorphism classes in the isogeny class $\AV(g^3)$ where $g=x^6 - x^5 + 2x^4 - 2x^3 + 4x^2 - 4x + 8$.
   Note that $g$ corresponds to a simple isogeny class of abelian varieties over $\F_2$.
   Define $K=\Q[x]/(g)$ and $\alpha =x \mod g$ and put $R=\Z[\alpha,\bar \alpha]$.
   The only over-order of $R$ is the maximal order $\cO_K$ of $K$ and, since $R$ is Gorenstein by \cite[Theorem 11]{CentelegheStix15} we get that $R$ is Bass.
   The Picard Group of $R$ is isomorphic to the cyclic group of order $3$ and it is generated by
   \[ I = 8R + \left(-32-11\alpha-\frac32\alpha^2-3\alpha^3-\frac34\alpha^4+\frac14\alpha^5\right)R.\]
   The maximal order $\cO_K$ is a principal ideal domain.
   Using Theorem \ref{thm:eqideals}\ref{thm:eqideals:b} we can count the isomorphism classes in $\AV(g^3)$, which are functorially represented by the following $R$-modules in $\cB(3)$:
   \begin{align*}
	& M_1=R \oplus R \oplus R\\
	& M_2=R \oplus R \oplus I\\
	& M_3=R \oplus R \oplus I^2\\
	& M_4=R \oplus R \oplus \cO_K\\
	& M_5=R \oplus \cO_K \oplus \cO_K\\
	& M_6=\cO_K \oplus \cO_K \oplus \cO_K.
   \end{align*}
   Using Proposition \ref{prop:homs} we can recover the endomorphism rings of the abelian varieties by their (functorial) representations as endomorphism rings of the modules $M_i$.
   For example, $\End(M_1)$ is the ring of $3\times 3$ matrices over $R$, while $\End(M_2)$ is the matrix ring
   \[ \begin{pmatrix}
         R & R & (R:I) \\
         R & R & (R:I) \\
         I & I & R
      \end{pmatrix}.\]
\end{example}

\begin{example}
  Let $g=(x^2 - 3x + 13)(x^2 + 6x + 13)$.
  Define $K=\Q[x]/(g)$ and $\alpha =x \mod g$ and put $R=\Z[\alpha,\bar \alpha]$.
  Using the algorithm described in \cite{MarICM18} we can compute the $6$ over-orders of $R$ and verify that they are all Gorenstein, that is, that $R$ is a Bass order.
  We apply Theorem \ref{thm:eqideals} to compute the number of isomorphism classes of abelian varieties in the isogeny class determined by $g^r$ as $r$ increases and collect the results for $1\leq r \leq 10$ in the following table.
\begin{center}
\begin{tabular}{|c|c|c|c|c|c|c|c|c|c|c|}
   \hline
   r 					& $1$ & $2$ & $3$ & $4$ & $5$ & $6$ & $7$ & $8$ & $9$ & $10$ \\ \hline
   &&&&&&&&&&\\[-10pt]
   $\#\left(\faktor{\AV(g^r)}{\simeq}\right)$ 	& $62$ & $97$ & $144$ & $206$ & $286$ & $387$ & $512$ & $664$ & $846$ & $1061$\\[5pt] \hline
\end{tabular}
\end{center}
\end{example}

\begin{example}
  In this example we will prove that certain isomorphism classes in the isogeny class $\AV(g^2)$, with  $g=x^4 - 2x^3 - 7x^2 - 22x + 121$ are not principally polarizable.
  Note that $g$ is irreducible and it corresponds to an ordinary isogeny class of abelian surfaces over $\F_{11}$.
  Define $K=\Q[x]/(g)$ and $\alpha =x \mod g$ and put $R=\Z[\alpha,\bar \alpha]$.
  The only over-order of $R$ is the maximal order $\cO_K$ of $K$ and, since $R$ is Gorenstein by \cite[Theorem 11]{CentelegheStix15} we get that $R$ is Bass.
  We computed that $\Pic(R)\simeq \Z/2\Z\times \Z/2\Z$ and $\Pic(\cO_K)\simeq \Z/2\Z$.
  One can verify that the classes of
  \[I=31R+(7+12\alpha-\alpha^2)R \text{ and } J=6734R+(2053-\alpha-\alpha^2)R \]
  generate $\Pic(R)$ and that the class of $J\cO_K$ is the generator of $\Pic(\cO_K)$.
  Using Theorem \ref{thm:eqideals} we can list the $8$ isomorphism classes of module in $\BassMod{2}$:
  \begin{align*}
  	M_1 &=R\oplus R	& M_5 &=R\oplus\cO_K \\
  	M_2 &=R\oplus I	& M_6 &=R\oplus J\cO_K \\
  	M_3 &=R\oplus J	& M_7 &=\cO_K\oplus\cO_K\\
  	M_4 &=R\oplus IJ	& M_8 &=\cO_K\oplus J\cO_K
  \end{align*}  
  Again using Theorem \ref{thm:eqideals} one can verify that all modules but $M_3$ and $M_4$ are self-dual, that is, isomorphic to their own dual. 
  Hence we can deduce that the abelian varieties corresponding via $\cF$ to $M_3$ and $M_4$ are not principally polarizable and hence cannot be Jacobian of curves.
  By \cite[Theorem 1.3]{Howe95} the isogeny class $\AV(g)$ is not principally polarizable, hence the abelian varieties in $\AV(g^2)$ do not admit a product polarization of degree $1$.
  On the other hand, in view of the fact that $\AV(g^2)$ becomes a $4$-th power of an elliptic curve over $\F_{11^2}$, the modules $M_i$ for $i\neq 3,4$ might still admit a principal polarization.
  As mentioned in the introduction, the behavior of the machinery developed with respect to field extensions will be investigated in a forthcoming paper. 
\end{example}

\begin{example}
  For all primes $p\leq 29$ and integers $0<r\leq 10$ we compute the number $N_{r,p}$ of isomorphism classes of abelian varieties over $\F_p$ that are isogenous to the $r$-th power of an elliptic curve over $\F_p$.
  Note that a characteristic polynomial $h$ of such an isogeny class cannot have real roots.
  Indeed if $h$ has a real roots then $x^2-p$ divides $h$ and hence an abelian surface with characteristic polynomial $(x^2-p)^2$, which is necessarily simple, would appear as an isogeny factor.
\begin{center}
\begin{tabular}{|c|c|c|c|c|c|c|c|c|c|c|}
   \hline
   $p$			   & $2$	   & $3$ & $5$ & $7$ & $11$ & $13$ & $17$ & $19$ & $23$ & $29$ \\ \hline   
   $N_{1,p}$       &$5$    &$8$    &$12$   &$18$   &$22$   &$32$   &$36$   &$42$   &$46$   &$60$   \\\hline
   $N_{2,p}$       &$5$    &$9$    &$14$   &$23$   &$25$   &$44$   &$44$   &$55$   &$53$   &$74$   \\\hline
   $N_{3,p}$       &$5$    &$10$   &$16$   &$28$   &$28$   &$58$   &$54$   &$68$   &$60$   &$90$   \\\hline
   $N_{4,p}$       &$5$    &$11$   &$18$   &$33$   &$31$   &$74$   &$66$   &$81$   &$67$   &$108$  \\\hline
   $N_{5,p}$       &$5$    &$12$   &$20$   &$38$   &$34$   &$92$   &$80$   &$94$   &$74$   &$128$  \\\hline
   $N_{6,p}$       &$5$    &$13$   &$22$   &$43$   &$37$   &$112$  &$96$   &$107$  &$81$   &$150$  \\\hline
   $N_{7,p}$       &$5$    &$14$   &$24$   &$48$   &$40$   &$134$  &$114$  &$120$  &$88$   &$174$  \\\hline
   $N_{8,p}$       &$5$    &$15$   &$26$   &$53$   &$43$   &$158$  &$134$  &$133$  &$95$   &$200$  \\\hline
   $N_{9,p}$       &$5$    &$16$   &$28$   &$58$   &$46$   &$184$  &$156$  &$146$  &$102$  &$228$  \\\hline
   $N_{10,p}$      &$5$    &$17$   &$30$   &$63$   &$49$   &$212$  &$180$  &$159$  &$109$  &$258$  \\\hline
\end{tabular}
\end{center}
\end{example}

\begin{example}
\label{ex:klein}
 Consider the Klein quartic $\cK$ over $\F_2$, which can be represented by the model
  \[ \cK: (X^2 + XZ)^2 + (X^2 + XZ)(Y^2 + YZ) + (Y^2 + YZ)^2 + Z^4 = 0. \]
 This model is equation $(1.22)$ in \cite{Elk99} reduced modulo $2$.
 It is known, see for example \cite[Section 3.3]{Elk99}, that $\cK$ is isogenous to $E^3$, where $E$ is an elliptic curve in the isogeny class determined by the Weil polynomial
 \[ h=x^2-x+2. \]
 Note that $\Z[x]/h$ is the maximal order of $K=\Q[x]/h$ and that $K$ has class number one.
 We deduce that $E$ is \emph{super-isolated}, that is, its isogeny class consists of only one isomorphism class.
 From Theorem \ref{thm:eqideals} we deduce that also $E^3$ is super-isolated and we conclude that the Jacobian $J(\cK)$ is isomorphic to $E^3$ (as an unpolarized abelian variety).
 Explicitly, $E$ is given by
 \[ E:Y^2 + (X + 1)Y = X^3 + 1. \]
 The fact that $J(\cK)$ is isomorphic to the cube of an elliptic curve could also be deduced, with some work, from the discussion on \cite[pp.~414-415]{Hoffmann91}.
 Clearly, such isomorphism is not compatible with the canonical principal polarization of $J(\cK)$, which is indecomposable, and the product principal polarization of $E^3$. 
\end{example}

\begin{example}
\label{ex:freq}
	In this example we are able to list all principally polarized abelian varieties in a particular isogeny class of surfaces together with their automorphisms.
	Consider the isogeny class of elliptic curves over $\F_3$ determined by the Weil polynomial 
	\[h=x^2+2x+3.\]
	The number field $K=\Q[x]/h$ has class number one and the order $\Z[x]/h$ is maximal.
	We deduce that the isogeny class $\AV(h)$ is super-isolated.
	We now consider the isogeny class $\AV(h^2)$ which is also super-isolated by Theorem \ref{thm:eqideals}. 
	From \cite[Table 3]{MaNart02} we find that $\AV(h^2)$ contains the Jacobian of the hyperelliptic curve 
	\[\cH: Y^2=X^6+X^4+X^2+1.\]
	From data on curves of genus at most three computed by Jonas Bergstr\"om
	in connection with the article \cite{BergFabervdG14} we can deduce that there are only two isomorphism classes of principally polarized abelian varieties in $\AV(h^2)$.
	The two isomorphism classes are represented by the Jacobian $J(\cH)$
	which has $48$ (polarized) automorphisms and the square of the elliptic curve
	\[ E:Y^2 = X^3 + X^2 + 2\]
	together with the product principal polarization, which has exactly $8$ polarized automorphisms (the canonical involutions on both factors and the involution coming from swapping the factors).
\end{example}


\newcommand{\etalchar}[1]{$^{#1}$}
\def\cprime{$'$}
\providecommand{\bysame}{\leavevmode\hbox to3em{\hrulefill}\thinspace}
\providecommand{\MR}{\relax\ifhmode\unskip\space\fi MR }
\providecommand{\MRhref}[2]{%
  \href{http://www.ams.org/mathscinet-getitem?mr=#1}{#2}
}
\providecommand{\href}[2]{#2}

\end{document}